\documentclass[12pt]{amsart}
\usepackage{amsmath,amscd,amssymb,amsfonts}
\usepackage[dvips]{graphicx,color} 
\usepackage[matrix, arrow]{xy}

\newcommand{\C}{{\mathbb C}}

\renewcommand{\H}{{\mathbb H}}

\renewcommand{\P}{{\mathbb P}}
\newcommand{\Q}{{\mathbb Q}}

\newcommand{\V}{{\mathbb V}}
\newcommand{\W}{{\mathbb W}}

\newcommand{\lra}{\longrightarrow }

\theoremstyle{plain}
\newtheorem{thm}{Theorem}
\newtheorem{lemma}[thm]{Lemma}
\newtheorem{cor}[thm]{Corollary}

\newtheorem{remark}[thm]{Remark}

\newtheorem{definition}[thm]{Definition}

\numberwithin{thm}{section}
\numberwithin{equation}{section}

\begin{document}

\title{Hodge numbers for the cohomology of Calabi-Yau type local systems}
\dedicatory{Klaus Hulek zum 60. Geburtstag gewidmet}
\author[H. Hollborn]{Henning Hollborn}
\author[S. M\"uller--Stach]{Stefan M\"uller--Stach}

\address{Mathematisches Institut der Johannes Gutenberg Universit\"at Mainz,
Staudingerweg 9, 55099 Mainz, Germany}
\email{hollborn@uni-mainz.de}
\email{mueller-stach@uni-mainz.de}

\begin{abstract}
We determine the Hodge numbers of the cohomology group 
$H^1_{L^2}(S,\V)=H^1(\bar S, j_*\V)$ using Higgs cohomology, where the local system $\V$ is induced by 
a family of Calabi-Yau threefolds over a smooth, quasi-projective curve $S$. This generalizes previous 
work to the case of quasi-unipotent, but not necessarily unipotent, local monodromies at infinity.
We give applications to Rohde's families of Calabi-Yau $3$-folds.
\end{abstract}

\footnotetext{Mathematics Classification Number: 14C25, 14F17, 14G35, 32M15}

\maketitle

\section*{Introduction} 

The first $L^2$-cohomology group $H^1_{L^2}(S,\V)=H^1(\bar{S},j_*\V)$, where 
$\V$ is a variation of Hodge structures $\V$ of weight $m$ over a smooth, quasi-projective curve $S=\bar{S} \setminus D {\buildrel j \over \hookrightarrow} \bar{S}$,
carries a pure Hodge structure of weight $m+1$ by \cite{zucker}. The goal of this paper is to continue the study of its Hodge numbers. 
We build up on the work done in \cite{hanoi}, using the methods of Zucker \cite{zucker}, but in addition the equivalent framework of Higgs bundles from 
the work of Jost, Yang, and Zuo \cite{jyz}. In \cite{hanoi} the local monodromies were assumed to be unipotent, but we show that one may skip this assumption, and
get similar formulae nevertheless. For simplicity, we will assume that all Hodge numbers of $\V$ are equal to one. Such situations
occur for families of elliptic curves, for the transcendental cohomology of families of K3 surfaces 
with generic Picard number $19$, and for certain families of Calabi-Yau 3-folds. 

The case of primary interest will be $m=3$, i.e., families of Calabi-Yau $3$-folds. 
However, for other applications we will also state results for the cases $m=1$ and $m=2$, which 
go back to work of Stiller \cite{stiller} and Cox-Zucker \cite{coxzucker}.

The group $H^1_{L^2}(S,\V)$ is of interest in theoretical physics \cite{morrison}, as the presence of codimension two cycles on the total space of a fibration
of Calabi-Yau $3$-folds implies that its $(2,2)$-Hodge number is non-zero.

The plan of this paper is as follows: After reviewing the basics of $L^2$-Higgs cohomology, 
we discuss the cases $m=1$, $m=2$ and $m=3$ separately and state the results in each case, comparing with the existing literature.
In case $m=3$ we extend the results from \cite{hanoi} to the case of non-unipotent monodromies at infinity and complete some tables of Hodge numbers there.   
In the last section we discuss some examples without maximally unipotent degeneration due to J.~C.~Rohde \cite{rohde,bert}. 
These examples are interesting as they contain many CM points in moduli induced by underlying Shimura varieties. 

\section{The basic set-up: Higgs cohomology} 

We consider a smooth, connected, projective family $f: X \longrightarrow S$ of $m$-dimensional varieties over a smooth quasi-projective curve $S$. 
Denote by $\bar{S}$ a smooth compactification of $S$, and by $\bar{f}: \bar{X} \to \bar{S}$ an extension of $f$ to a flat family over $\bar{S}$. 
Associated to this situation is a local system $\V=R^mf_*\C$ and the corresponding vector bundle $V:=\V \otimes {\mathcal O}_S$
on $S$. We would like to compute $H^1(\bar{S},j_*\V)$ in terms of the degeneration data of $\bar{f}$. 

We denote by $T$ the local monodromy matrix around a point in $D$ at infinity. 
$\V$ has quasi-unipotent monodromies at all points of $D:=\bar{S} \setminus S$. 
If $\bar{f}$ is semistable in codimension one, then the local monodromies are unipotent. 
After Deligne, the vector bundle $V$ has a quasi-canonical extension $\bar{V}$ to $\bar{S}$ 
as a vector bundle together with a logarithmic Gau{\ss}-Manin connection
$$
\bar{\nabla}: \bar{V} \to \bar{V} \otimes \Omega^1_{\bar{S}}(\log D).
$$
In the case of unipotent local monodromies $\bar{V}$ has degree zero, but not in the general case. 
The Hodge filtration $V=F^0 \supset F^1 \supset \cdots \supset F^m \supset F^{m+1}=0$ also extends to $\bar{S}$ and we define
$$
E^{p,m-p}:=F^p/F^{p+1}
$$
as vector bundles on $\bar{S}$. Let 
$$
E:= \bigoplus_{p=0}^m  E^{p,m-p}
$$
be the associated Higgs bundle with Higgs field
$$
\vartheta: E \to E \otimes \Omega^1_{\bar{S}}(\log D),
$$
where 
$$
\vartheta: E^{p,m-p} \to E^{p-1,m-p+1} \otimes \Omega^1_{\bar{S}}(\log D)
$$ 
is induced by $\bar{\nabla}$ and Griffiths transversality. In particular, 
the Higgs bundle induces a complex of vector bundles 
$$
E^\bullet: E {\buildrel \vartheta \over \longrightarrow} E \otimes \Omega^1_{\bar{S}}(\log D).
$$
Since $\dim(S)=1$ here, the usual condition $\vartheta \wedge \vartheta=0$ is empty, and the complex lives only in degrees $0$ and $1$.
The hypercohomology group $\H^1(E^\bullet)$ computes $H^1(S,\V)$ \cite{jyz,zucker}. 

If the local monodromy matrix $T$ at some point $P \in D$ is unipotent, then its logarithm $N:=\log(T)$ is nilpotent. Any nilpotent endomorphism $N$ 
of a vector space $V_0$ satisfying $N^m \neq 0$ and $N^{m+1}=0$ defines a natural increasing filtration on $V_0$:
\[
0 \subset W_{-m} \subset W_{-m+1} \subset \cdots \subset W_0 \subset W_{1} \subset \cdots \subset W_{m}=V_0, 
\]
which has the following definition: if $N^{m+1}=0$ but $N^m \neq 0$, we put
\[
W_{m-1}={\rm Ker}(N^m), \quad W_{-m}={\rm Im}(N^m).
\]
The further groups $W_k$ for $-m < k \le m-2$ are inductively constructed by
requiring that $N(W_k)={\rm Im}(N) \cap W_{k-2} \subset W_{k-2}$ and 
$$
N^k: {\rm Gr}^W_{k}(V_0) \to {\rm Gr}^W_{-k}(V_0) 
$$
are isomorphisms. If $V_0$ is the fiber of $\V$ at a smooth point this filtration is called the 
{\em monodromy weight filtration}. Prop. 4.1. of \cite{zucker} states that in the unipotent case one has a resolution which locally looks like
$$
0 \to j_*\V \to [W_0 + t \bar{V}] {\buildrel \bar{\nabla} \over \longrightarrow} \frac{dt}{t} \otimes [W_{-2}+t \bar{V}] \to 0. 
$$
In the quasi-unipotent case with no unipotent part one has on the other hand locally a resolution of the form
$$
0 \to j_*\V \to \bar{V} {\buildrel \bar{\nabla} \over \longrightarrow} \frac{dt}{t} \otimes \bar{V} \to 0
$$
by Prop. 6.9. of \cite{zucker} and the stalk of $j_*\V$ at $t=0$ is zero. 

Zucker also studies the Hodge filtration on $\bar{V}$. Theorem 11.6 in loc. cit. gives eventually a representation
of $H^1(\bar{S}, j_*\V)$ and its Hodge components. Instead of this de Rham representation we will switch to the corresponding
Higgs version. 

We can use the monodromy weight filtration $W_*$ on each fiber $E_P,\; P \in D$ to define the {\em $L^2$-Higgs complex} 
\[ 
(\Omega_{(2)}^{\bullet}(E),\theta): 
\Omega_ {(2)}^0(E) \subset E,\;\;\;\Omega_{(2)}^1(E) \subset E \otimes \Omega_{\bar S}^1(\log D),
\]
The sub-sheaves in each degree are defined near $P \in D$ as
\[
\Omega^0_{(2)}(E):=W_0 + tE, \quad \Omega^1_{(2)}(E):=(W_{-2} + tE) \otimes \Omega^1_{\bar S}(\log D).
\]
The notation is such that $t$ is a local parameter with $P=\{t=0\} \in D$ and
the monodromy weight filtration is given by the logarithm $N=\log(T)$. 
At any point $P \in S$ outside $D$, the $L^2$-Higgs complex is just given by the Higgs bundle.

It can be shown \cite{jyz} that the hypercohomology of the $L^2-$ Higgs complex 
$(\Omega_{(2)}^{\bullet}(E),\theta)$ is isomorphic to the $L^2$-cohomology group
\[ 
H^{k}_{(2)}(S,E)=H^k(\bar{S},j_*\V)=\H^k(\Omega_{(2)}^{\bullet}(E),\theta).
\]
In the following sections, we study the local structure of $(\Omega_{(2)}^{\bullet},\theta)$ for the case of logarithmic Higgs bundles of type $(1,1,\ldots,1,1)$, so that
each summand $E^{p,m-p}$ of $E$ is a line bundle.
For the points $P \in D$ one has to distinguish cases corresponding to the possible Jordan normal forms of the endomorphism $N$.
The decomposition 
$$
E= \bigoplus_{p=0}^m  E^{p,m-p}
$$
induces a decomposition 
$$
\Omega_{(2)}^{\bullet}(E)=\bigoplus_{p=0}^m \Omega_{(2)}^{\bullet}(E)^{p.m-p},
$$
where
$$
\Omega_{(2)}^0(E)^{p,m-p}:=\Omega_{(2)}^0(E) \cap E^{p,m-p},
$$
$$
\Omega_{(2)}^1(E)^{p,m-p}:=\Omega_{(2)}^1(E) \cap E^{p-1,m-p+1} \otimes \Omega_{\bar S}^1(\log D).
$$
The hypercohomology spectral sequence associated to this filtration induces the Hodge structure on $H^{k}_{(2)}(S,E)$.  

\section{Elliptic families} 

In the case of families of elliptic curves ($m=1$) we obtain from the previous results: 

\begin{thm}[Zucker \cite{zucker}]
The $L^2$-Higgs complex for $E$ is given by:
\begin{align*}
\Omega^0_{(2)}(E)^{1,0}&=E^{1,0}(-I)  \cr 
\Omega^0_{(2)}(E)^{0,1}&=E^{0,1}  \cr 
\Omega^1_{(2)}(E)^{1,0}&=E^{1,0}(II) \otimes \Omega^1_{\bar{S}}  \cr 
\Omega^1_{(2)}(E)^{0,1}&=E^{0,1}(II) \otimes \Omega^1_{\bar{S}}
\end{align*}
Here $I$ is the set of points with unipotent local monodromy (denoted by type $I_b$ in the Kodaira classification of singular fibers), $II$ the set of remaining
non-unipotent singular points. 
\end{thm}

\begin{proof}
Elliptic fibrations have either unipotent local monodromy $T$ at points of type $I$, where the Jordan normal form of $T$ 
is given by the matrix 
$$
T=\left( \begin{matrix} 1 & 1 \cr 0 & 1 \end{matrix} \right)
$$
or non-unipotent local monodromies, where $T$ is equivalent to
$$
T=\left( \begin{matrix} \lambda & 1 \cr 0 & \lambda \end{matrix} \right), \; {\rm or} \; 
T=\left( \begin{matrix} \lambda_1 & 0 \cr 0 & \lambda_2 \end{matrix} \right).
$$
for some roots of unity $\lambda, \lambda_i \neq  1$.
In the first case, Zucker \cite[Prop. 4.1.]{zucker} gives a monodromy weight filtration locally at a point $P=\{t=0\} \in I$ which looks like 
$W_0=W_{-1}= t E^{1,0} \oplus E^{0,1}$ and $W_{-2}=tE$, hence the claim. At a non-unipotent point  $P \in II$, \cite[ Prop. 6.9.]{zucker}
shows the claim as well.  
\end{proof}

These observations imply the following well-known theorem.

\begin{thm}[Cox-Zucker \cite{coxzucker}] \label{m1} Assume that $\V$ is irreducible, 
and that $\vartheta: E^{1,0} \to E^{0,1}\otimes \Omega^1_ {\bar{S}}(\log D)$ 
is a non-zero map with $a+|II|>0$, where $a:=\deg E^{1,0}$. 
Then the Hodge numbers for the pure Hodge structure of weight $2$ on $H^1(\bar{S},j_*\V)$ are 
$$
h^{2,0}=h^{0,2}=g-1+a+ |II|, \; h^{1,1}=2g-2-2a+|I|.
$$ 
This implies the well-known formula $h^1(j_*\V)=4g-4 + |I| +2 |II|$, see \cite[page 39]{coxzucker}.
\end{thm}

\begin{proof} The Higgs complex is given by
$$
\begin{xy}
\xymatrix{
E^{1,0}(-I)  \ar[rd]^{\vartheta}_{\neq 0} &   E^{0,1}   \\
E^{1,0}(II)\otimes  \Omega^1_{\bar{S}} &  E^{0,1}(II) \otimes  \Omega^1_{\bar{S}}  } 
\end{xy}
$$
Note that both $\Omega^0_{(2)}(E)^{0,1}=E^{0,1}$ and $\Omega^1_{(2)}(E)^{1,0}=E^{1,0}(II) \otimes \Omega^1_{\bar{S}}$ 
have neither incoming nor outgoing Higgs differential. By Hodge duality, i.e., $h^{2,0}=h^{0,2}$, we get
$h^1(E^{0,1})=h^0(E^{1,0}(II)\otimes \Omega^1_ {\bar{S}})$. Under the assumption $a+|II|>0$ this 
gives the formula for $h^{2,0}=h^{0,2}$ by applying Riemann-Roch to the line bundle $E^{1,0}(II)\otimes \Omega^1_ {\bar{S}}$. 
$h^{1,1}$ is $h^0$ of the cokernel of $\vartheta: \Omega^0_{(2)}(E)^{1,0} \to \Omega^1_{(2)}(E)^{0,1}$, 
hence the difference of the degrees of both line bundles, from which the rest of the assertion follows.
\end{proof}

\begin{remark} \label{pardeg} The assumptions in the theorem are not independent. The condition that $a+|II|>0$ is not always 
satisfied, but in many cases: the parabolic degree of any subbundle $F \subset E$ is defined as 
$$
\deg_p F :=\deg F + \sum_{P \in II} \; \sum_{0 \le \alpha < 1} \alpha \dim ({\rm Gr}_\alpha F_P),
$$
where ${\rm Gr}_\alpha$ is the graded piece of the parabolic filtration corresponding to the monodromy $\exp(2 \pi i \alpha)$. 
For $F=E^{0,1}$, a Higgs subbundle of $(E,\vartheta)$ with $\vartheta=0$, 
one gets $\deg_p(E^{0,1}) \le \deg_p(E)=0$ by the Simpson correspondence \cite[Prop. 2.1]{jz}, which implies $\deg_p(E^{1,0})=-\deg_p(E^{0,1}) \ge 0$. 
Therefore, if the double sum is not zero, i.e., some $\alpha>0$ occurs, then $0 \le \deg_p E^{1,0} < a + |II|$, since all Hodge numbers are $1$. 
\end{remark}

\begin{remark}
In the case $\bar{S}=\P^1$ and $\deg E^{0,1} \le -2$,
the proof states that $h^1(E^{0,1})=h^0((E^{0,1})^\vee \otimes \Omega^1_ {\bar{S}})=h^0(E^{1,0}(II)\otimes \Omega^1_ {\bar{S}})$. 
This implies that $E^{0,1}=(E^{1,0})^{-1}(-II)$. 
\end{remark}

\section{Families of K3 surfaces} 

With the previous notation, we consider a smooth projective family of K3 surfaces $f: X \longrightarrow S$ with generic Picard number $19$ 
over a smooth curve $S$. 
Associated to this situation is a local system $\V \subset R^2f_*\C$ of rank three, given fiberwise by the transcendental cohomology. Let 
$$
E:= E^{2,0} \oplus E^ {1,1} \oplus  E^{0,2}
$$
be the associated Higgs bundle with Higgs field
$$
\vartheta: E \to E \otimes \Omega^1_{\bar{S}}(\log D).
$$
Now we make the following \\
{\bf Assumption:} \emph{Each local monodromy is either unipotent or has no unipotent part.}
In other words, there are no mixed cases with non-zero unipotent and non--unipotent pieces.
This implies that the Jordan normal forms for the local monodromies are 
$$
T=\left( \begin{matrix} 1 & 1 & 0  \cr 0 & 1 & 1 \cr 0 & 0 & 1 \end{matrix} \right),  \; 
\left( \begin{matrix} 1 & 1 & 0  \cr 0 & 1 & 0  \cr  0 & 0 & 1 \end{matrix} \right),  \;
\left( \begin{matrix} \lambda & 1 & 0  \cr 0 & \lambda & 1 \cr 0 & 0 & \lambda \end{matrix} \right), \;
\left( \begin{matrix} \lambda_1 & 0  & 0 \cr 0 & \lambda_2 &  0 \cr 0 & 0 & \lambda_3  \end{matrix} \right) \; {\rm or} \;
\left( \begin{matrix} \lambda_1 & 1  & 0 \cr 0 & \lambda_1 &  0 \cr 0 & 0 & \lambda_2  \end{matrix} \right),
$$
with $\lambda, \lambda_i \neq 1$ roots of unity. 

\begin{lemma}
Only the Jordan normal forms 
$$
T_I=\left( \begin{matrix} 1 & 1 & 0  \cr 0 & 1 & 1 \cr 0 & 0 & 1 \end{matrix} \right), \;  
T_{II}=\left( \begin{matrix} \lambda & 1 & 0  \cr 0 & \lambda & 1 \cr 0 & 0 & \lambda \end{matrix} \right), \;
T_{II}=\left( \begin{matrix} \lambda_1 & 0  & 0 \cr 0 & \lambda_2 &  0 \cr 0 & 0 & \lambda_3  \end{matrix} \right) \; {\rm or} \;
T_{II}=\left( \begin{matrix} \lambda_1 & 1  & 0 \cr 0 & \lambda_1 &  0 \cr 0 & 0 & \lambda_2  \end{matrix} \right) 
$$
with $\lambda, \lambda_i \neq 1$ occur. The case $I$ is unipotent, the cases $II$ are strictly quasi--unipotent. 
\end{lemma}

\begin{proof} In the unipotent case, as in \cite[p. 11]{hanoi}, both maps in the sequence
$$
E^{2,0} {\buildrel N \over \to} E^{1,1} {\buildrel N \over \to} E^{0,2}
$$
are dual to each other. Hence, if $N^2=0$, both must be zero, which implies $N=0$. This excludes the second matrix.  
\end{proof}

\begin{thm}
The $L^2$-Higgs complex for $E$ is given by:
\begin{align*}
\Omega^0_{(2)}(E)^{2,0}&=E^{2,0}(-I)  \cr 
\Omega^0_{(2)}(E)^{1,1}&=E^{1,1}  \cr 
\Omega^0_{(2)}(E)^{0,2}&=E^{0,2}  \cr 
\Omega^1_{(2)}(E)^{2,0}&=E^{2,0}(II) \otimes \Omega^1_{\bar{S}} \cr 
\Omega^1_{(2)}(E)^{1,1}&=E^{1,1}(II) \otimes \Omega^1_{\bar{S}} \cr 
\Omega^1_{(2)}(E)^{0,2}&=E^{0,2}(I+II) \otimes \Omega^1_{\bar{S}}
\end{align*}
Here $I$ is again the set of points with unipotent local monodromy, $II$ the set of remaining
non-unipotent singular points. 
\end{thm}

\begin{proof}
The proof is exactly as in the case $m=1$ using \cite[Props. 4.1 and 6.9.]{zucker}.  
\end{proof}

\begin{thm} \label{m2} Assume that $\V$ is irreducible, and that 
$\vartheta: E^{2,0} \to E^{1,1}\otimes \Omega^1_ {\bar{S}}(\log D)$ as well as 
$\vartheta: E^{1,1} \to E^{0,2}\otimes \Omega^1_ {\bar{S}}(\log D)$ are non-zero maps with $a+|II|>0$, 
where $a:=\deg E^{2,0}$. 
Then the Hodge numbers for the pure Hodge structure of weight $3$ on $H^1(\bar{S},j_*\V)$ are 
$$
h^{3,0}=h^{0,3}=g-1+a+|II|, \; h^{2,1}=h^{1,2}=2g-2-a+|I|+\frac{1}{2}|II|.
$$ 
In total, one has $h^1(j_*\V)=6g-6 + 2|I|+ 3|II|$, which agrees with \cite[Prop 3.6.]{hanoi}.
\end{thm}

\begin{proof} The Higgs complex is given by
$$
\begin{xy}
\xymatrix{
E^{2,0}(-I)  \ar[rd]^{\vartheta}_{\neq 0} &   E^{1,1}  \ar[rd]^{\vartheta}_{\neq 0}  & E^{0,2}  \\
E^{2,0}(II)\otimes  \Omega^1_{\bar{S}} &  E^{1,1}(II) \otimes  \Omega^1_{\bar{S}} & E^{0,2}(I+II)\otimes  \Omega^1_{\bar{S}} } 
\end{xy}
$$
Note that both $\Omega^0_{(2)}(E)^{0,2}=E^{0,2}$ and $\Omega^1_{(2)}(E)^{2,0}=E^{2,0}(+II) \otimes \Omega^1_{\bar{S}}$ 
have neither incoming nor outgoing Higgs differential. Hodge duality, i.e., $h^{3,0}=h^{0,3}$, 
implies $h^0(E^{2,0}(II)\otimes \Omega^1_{\bar{S}})=h^1(E^{0,2})$. 
Riemann-Roch applied to $E^{2,0}(II)$ then gives the formula for $h^{3,0}=h^{0,3}$ under the assumption $a+|II|>0$. 

The space $H^{2,1}$ is represented as global sections of the cokernel of the map 
$$
\Omega_{(2)}^0(E)^{2,0}\stackrel{\theta}{\lra} \Omega_{(2)}^1(E)^{1,1},
$$ 
hence we have to count the zeros of a map of line bundles
\[E^{2,0}(-I) \lra E^{1,1}(+II) \otimes \Omega_{\bar S}^1.\]
This number is given by the difference in degrees of the line bundles, so
\[
h^{2,1}=h^{1,2}=\deg E^{1,1}(+II) \otimes \Omega_{\bar S}^1-\deg E^{2,0}(-I)=2g-2 +\deg E^{1,1}-a+|I|+|II|.
\]
It is not true that $\deg E^{1,1}=0$ in the non--unipotent case.
Indeed let $b:=\deg E^{1,1}$. We thus obtain $h^{2,1}= 2g-2+b-a+|I|+|II|$. 

Now we use a checking sum: By \cite[Prop 3.6.]{hanoi} we know that 
$$
h^1(j_*\V)=h^{3,0}+h^{2,1}+h^{1,2}+h^{0,3}=6g-6+2|I|+3|II|,
$$
since by our assumption non--unipotent local monodromies have zero invariant subspace. 
This implies that $b=-\frac{1}{2}|II|$.
\end{proof}

\begin{remark} Condition $a+|II|>0$ again follows in many cases, see Remark \ref{pardeg}. Assume that $\bar{S}=\P^1$ and that $\deg E^{0,2} \le -2$.
The proof states that $h^1(E^{0,2})=h^0((E^{0,2})^\vee \otimes \Omega^1_ {\bar{S}})=h^0(E^{2,0}(II)\otimes \Omega^1_ {\bar{S}})$. 
This implies that $E^{0,2}=(E^{2,0})^{-1}(-II)$. 
\end{remark}

\section{Families of Calabi--Yau $3$--folds} 

We consider a smooth projective family of Calabi--Yau $3$--folds $f: X \longrightarrow S$ 
over a smooth curve $S$ as in \cite{hanoi}. 
Assume that $\bar{S}$ is a smooth compactification and consider a real VHS  
$\V \subset R^3f_*\C$ of rank four with Hodge numbers $(1,1,1,1)$. We use the previous notation and make again the assumption that 
each local monodromy is either unipotent or has no unipotent part.

This implies that the Jordan forms for the local monodromies $T$ are 
\begin{small}
$$
\left( \begin{matrix} 1 & 0 & 0 & 0  \cr 0 & 1 & 1 & 0 \cr 0 & 0 & 1 & 0 \cr 0 & 0 & 0 & 1 \end{matrix} \right), \; 
\left( \begin{matrix} 1 & 1 & 0 & 0  \cr 0 & 1 & 0 & 0 \cr 0 & 0 & 1 & 1 \cr 0 & 0 & 0 & 1 \end{matrix} \right), \; 
\left( \begin{matrix} 1 & 1 & 0 & 0  \cr 0 & 1 & 1 & 0 \cr 0 & 0 & 1 & 0 \cr 0 & 0 & 0 & 1 \end{matrix} \right), \; 
\left( \begin{matrix} 1 & 1 & 0 & 0  \cr 0 & 1 & 1 & 0 \cr 0 & 0 & 1 & 1 \cr 0 & 0 & 0 & 1 \end{matrix} \right), \;
\left( \begin{matrix} \lambda_1 & 0  & 0 & 0 \cr 0 & \lambda_2 &  0 & 0 \cr 0 & 0 & \lambda_3 & 0 \cr 0 & 0 & 0 &  \lambda_4 \end{matrix} \right),
$$
$$
\left( \begin{matrix} \lambda_1 & 0 & 0 & 0  \cr 0 & \lambda_2 & 1 & 0 \cr 0 & 0 & \lambda_2 & 0 \cr 0 & 0 & 0 & \lambda_3 \end{matrix} \right), \; 
\left( \begin{matrix} \lambda_1 & 1 & 0 & 0  \cr 0 & \lambda_1 & 0 & 0 \cr 0 & 0 & \lambda_2 & 1 \cr 0 & 0 & 0 & \lambda_2 \end{matrix} \right), \;
\left( \begin{matrix} \lambda_1 & 1 & 0 & 0  \cr 0 & \lambda_1 & 1 & 0 \cr 0 & 0 & \lambda_1 & 0 \cr 0 & 0 & 0 & \lambda_2 \end{matrix} \right) \; \; {\rm or}
\left( \begin{matrix} \lambda & 1 & 0 & 0  \cr 0 & \lambda & 1 & 0 \cr 0 & 0 & \lambda & 1 \cr 0 & 0 & 0 & \lambda \end{matrix} \right),
$$
\end{small}
with $\lambda, \lambda_i \neq 1$ roots of unity. 

\begin{lemma}
Only the Jordan normal forms 
\begin{tiny}
$$
T_I= \left( \begin{matrix} 1 & 0 & 0 & 0  \cr 0 & 1 & 1 & 0 \cr 0 & 0 & 1 & 0 \cr 0 & 0 & 0 & 1 \end{matrix} \right), \; 
T_{II}= \left( \begin{matrix} 1 & 1 & 0 & 0  \cr 0 & 1 & 0 & 0 \cr 0 & 0 & 1 & 1 \cr 0 & 0 & 0 & 1 \end{matrix} \right),\; 
T_{III}= \left( \begin{matrix} 1 & 1 & 0 & 0  \cr 0 & 1 & 1 & 0 \cr 0 & 0 & 1 & 1 \cr 0 & 0 & 0 & 1 \end{matrix} \right), \;
T_{IV} = \left( \begin{matrix} \lambda_1 & 0  & 0 & 0 \cr 0 & \lambda_2 &  0 & 0 \cr 0 & 0 & \lambda_3 & 0 \cr 0 & 0 & 0 &  \lambda_4 \end{matrix} \right),
$$
$$
T_{IV} = \left( \begin{matrix} \lambda_1 & 0 & 0 & 0  \cr 0 & \lambda_2 & 1 & 0 \cr 0 & 0 & \lambda_2 & 0 \cr 0 & 0 & 0 & \lambda_3 \end{matrix} \right), \; 
T_{IV} = \left( \begin{matrix} \lambda_1 & 1 & 0 & 0  \cr 0 & \lambda_1 & 0 & 0 \cr 0 & 0 & \lambda_2 & 1 \cr 0 & 0 & 0 & \lambda_2 \end{matrix} \right), \;
T_{IV} = \left( \begin{matrix} \lambda_1 & 1 & 0 & 0  \cr 0 & \lambda_1 & 1 & 0 \cr 0 & 0 & \lambda_1 & 0 \cr 0 & 0 & 0 & \lambda_2 \end{matrix} \right) \; {\rm or} \;
T_{IV} = \left( \begin{matrix} \lambda & 1 & 0 & 0  \cr 0 & \lambda & 1 & 0 \cr 0 & 0 & \lambda & 1 \cr 0 & 0 & 0 & \lambda \end{matrix} \right)
$$
\end{tiny}
with $\lambda, \lambda_i \neq 1$ occur. The cases $I$, $II$ and $III$ are unipotent, the cases $IV$ are strictly quasi--unipotent. 
\end{lemma}

\begin{proof}
See the discussion of normal forms in \cite[Sect. 1]{hanoi}.
\end{proof}

\begin{thm}
The $L^2$-Higgs complex for $E$ is given by:
\begin{align*}
\Omega^0_{(2)}(E)^{3,0}&=E^{3,0}(-II-III)  \cr 
\Omega^0_{(2)}(E)^{2,1}&=E^{2,1}(-I-III)  \cr 
\Omega^0_{(2)}(E)^{1,2}&=E^{1,2}(-II)  \cr 
\Omega^0_{(2)}(E)^{0,3}&=E^{0,3}  \cr 
\Omega^1_{(2)}(E)^{3,0}&=E^{3,0}(IV) \otimes \Omega^1_{\bar{S}}  \cr 
\Omega^1_{(2)}(E)^{2,1}&=E^{2,1}(IV) \otimes \Omega^1_{\bar{S}}  \cr 
\Omega^1_{(2)}(E)^{1,2}&=E^{1,2}(IV) \otimes \Omega^1_{\bar{S}} \cr
\Omega^1_{(2)}(E)^{0,3}&=E^{0,3}(III+IV)) \otimes \Omega^1_{\bar{S}}
\end{align*}
Here $I$, $II$, $III$ are again the sets of points with unipotent local monodromy, $IV$ the set of remaining
non-unipotent singular points. 
\end{thm}

\begin{proof}
The proof is exactly as in the cases $m=1$ and $m=2$ using \cite[Props. 4.1 and 6.9.]{zucker}. 
\end{proof}

In summary, we get the following result, which agrees with \cite[Prop 3.6.]{hanoi} in the unipotent case.

\begin{thm} \label{m3} Assume that $\V$ is irreducible, and that 
$\vartheta: E^{3,0} \to E^{2,1}\otimes \Omega^1_ {\bar{S}}(\log D)$ as well as 
$\vartheta: E^{2,1} \to E^{1,2}\otimes \Omega^1_ {\bar{S}}(\log D)$
and $\vartheta: E^{1,2} \to E^{0,3}\otimes \Omega^1_ {\bar{S}}(\log D)$ are non-zero maps with $a+|IV|>0$, 
where $a:=\deg E^{3,0}$ and $b:=\deg E^{2,1}$.  
Then the Hodge numbers for the pure Hodge structure of weight $4$ on $H^1(\bar{S},j_*\V)$ are 
$$
h^{4,0}=h^{0,4}=g-1+a+|IV|, \; h^{3,1}=h^{1,3}=2g-2+b-a+|II|+|III|+|IV|, 
$$
$$ 
h^{2,2}= |I|+|III|-2b+2g-2. 
$$ 
In total, one has 
$$
h^1(j_*\V)=8g-8 + |I| + 2|II| + 3|III|+ 4|IV|.
$$
\end{thm}

\begin{proof} The Higgs complex is given by
$$
\begin{xy}
\xymatrix{
E^{3,0}(-II-III)  \ar[rd]^{\vartheta}_{\neq 0} &   E^{2,1}(-I-III)  \ar[rd]^{\vartheta}_{\neq 0}  & E^{1,2}(-II) \ar[rd]^{\vartheta}_{\neq 0}  &  E^{0,3} \\
E^{3,0}(IV)\otimes  \Omega^1_{\bar{S}} &  E^{2,1}(IV) \otimes  \Omega^1_{\bar{S}} & E^{1,2}(IV)\otimes  \Omega^1_{\bar{S}} &  E^{0,3}(III+IV) \otimes  \Omega^1_{\bar{S}} } 
\end{xy}
$$
Note that both $\Omega^0_{(2)}(E)^{0,3}=E^{0,3}$ and $\Omega^1_{(2)}(E)^{3,0}=E^{3,0}(IV) \otimes \Omega^1_{\bar{S}}$ 
have neither incoming nor outgoing Higgs differential. Hodge duality, i.e., $h^{4,0}=h^{0,4}$, 
implies $h^0(E^{3,0}(IV)\otimes \Omega^1_{\bar{S}})=h^1(E^{0,3})$. 
Riemann-Roch applied to $E^{3,0}(IV)$ then gives the formula for $h^{4,0}=h^{0,4}$ under the assumption $a+|IV|>0$. 

As in \cite{hanoi} the space $H^{3,1}$ is represented as global sections of the cokernel of the map 
$\Omega_{(2)}^0(E)^{3,0}\stackrel{\theta}{\lra} \Omega_{(2)}^1(E)^{2,1}$, hence we have to count the zeros 
of a map of line bundles
\[E^{3,0}(-II-III) \lra E^{2,1}(IV) \otimes \Omega_{\bar S}^1.\]
This number is therefore given by the difference in degrees of the line bundles, i.e., 
\[h^{3,1}=h^{1,3}=\deg E^{2,1}(IV) \otimes \Omega_{\bar S}^1-\deg E^{3,0}(-II-III)=b+2g-2-a+|II|+|III|+|IV|.\]

In a similar way, $H^{2,2}$ is represented as global sections of the cokernel of the map 
$\Omega_{(2)}^{0}(E)^{2,1} \stackrel{\theta}{\lra} \Omega_{(2)}^1(E)^{1,2}$, hence we have to 
count the zeros of the map of line bundles
\[ E^{2,1}(-I-III) \lra E^{1,2}(+IV) \otimes \Omega_{\bar S}^1. \]
\end{proof}

\begin{remark} \label{degreeremark} Condition $a+|IV|>0$ again follows in many cases, see Remark \ref{pardeg}. 
Assume that $\bar{S}=\P^1$ and that $\deg(E^{0,3})\le -2$.
The proof states that $h^1(E^{0,3})=h^0((E^{0,3})^\vee \otimes \Omega^1_ {\bar{S}})=h^0(E^{3,0}(IV)\otimes \Omega^1_ {\bar{S}})$. 
This implies that $E^{0,3}=(E^{3,0})^{-1}(-IV)$. Hence, if $a':=-\deg E^{0,3}$, one has $a'=a+|IV|$.

It is not clear that $\deg E^{1,2}=-\deg E^{2,1}$. 
Indeed let $b':= - \deg E^{1,2}$. We obtain $h^{2,2}= |I|+|III|+|IV|-b-b'+2g-2$. 

Now we use a checking sum: By \cite[Prop 3.6.]{hanoi} we know that 
$$
h^1(j_*\V)=h^{4,0}+h^{3,1}+h^{2,2}+h^{1,3}+h^{0,4}=8g-8 + |I| + 2|II| + 3|III|+ 4|IV|,
$$
since by our assumption non--unipotent local monodromies have zero invariant subspace. 
This implies that $b'=b+|IV|$.
\end{remark}

Using the formulas obtained above, one can revisit the tables for Hodge numbers in \cite{hanoi} and add the degrees $a$ and $b$ of the Hodge bundles (see table). 
In the table, $e$ is the degree of a covering map $\P^1 \to \P^1$ of the form $z \mapsto z^e$ ramified in $0$ and $\infty$. 
The numbering follows the database \cite{AESZ}.

In the following sections, we need in addition the following upper bound for $a$ from the work of Jost and Zuo:

\begin{thm}\cite[Theorem 1]{jz} \label{jzthm} 
\[
\deg E^{3,0} \le \left( \frac{1}{2} (h^{2,1}-h^{2,1}_0) + (h^{3,0}-h^{3,0}_0)\right)(2g-2+\sharp D),  
\]
where a subscript $0$ denotes the kernel of $\vartheta$. More generally, if $\V$ is a real VHS of odd weight $k=2l+1 \ge 1$, then one has
\[
\deg E^{k,0} \le  \left( \frac{1}{2} (h^{k-l,l}-h^{k-l,l}_0) + \sum_{j=0}^ {l-1} (h^{k-j,j}-h^{k-j,j}_0)\right)(2g-2+\sharp D).
\]
 
\end{thm}

If we assume that all maps $\vartheta$ are non-zero (except the one on $E^{0,3}$), and all ranks $h^{p,q}$ are $1$ as in our case, 
then the inequality simply becomes:
\[
\deg E^{3,0} \le \frac{3}{2} (2g-2+\sharp D).
\]
In the case $\bar{S}=\P^1$ we therefore obtain $\deg E^{3,0} \le \frac{3}{2} (\sharp D-2)$. 
In the case of $3$ singular points, we get $\deg E^{3,0} \le \frac{3}{2}$, hence $a=\deg E^{3,0} \le 1$.

\newpage 

\begin{small}
\[
\begin{array}{|r|l|l||c|c|c|c|c|c|c|}
\hline
\#&\textup{Model}&T_\infty&e&h^1(j_*\V)&h^{4,0}&h^{3,1}&h^{2,2}&a&b\\
\hline
1& \P^4[5]  & IV &1&0&0&0&0&0&0 \\
&           &  &2&1&0&0&1&0&0 \\
&           &  &5&0&0&0&0&1&2\\
&           &  &10&1&1&1&1&2&4\\
\hline
2& \P(1,1,1,2,5)[10]          & IV  &1&0&0&0&0&0&0\\
&           &  &2&1&0&0&1&0&0\\
&           &  &5&4&0&0,1,2&4,2,0&0&0,1,2\\
&           &  &10&5&0&0,1,2&5,3,1&1&2,3,4\\
\hline
3& \P^7[2,2,2,2]  & III  &1&0&0&0&0&0&0\\
&          &   &2&0&0&0&0&1&1\\
&          &   &2k&2k-2&k-1&0&0&k&k\\
\hline   
4&    \P^5[3,3]     & II    &1&0&0&0&0&0&0\\
&       &      &2&1&0&0&1&0&0\\
&       &      &3&0&0&0&0&1&1\\
&       &      &6&3&1&0&1&2&2\\
\hline
5&   \P^6[2,2,3]       & I   &1&0&0&0&0&0&0\\
&           &  &6&2&0&0\mbox{ or }1&2\mbox{ or }0&1&2\mbox{ or }3\\
\hline               
6& \P^5[2,4]        & I    &1&0&0&0&0&0&0\\
&          &   &4&0&0&0&0&1&2\\
&          &   &8&4&1&1&0&2&4\\
\hline
7& \P(1,1,1,1,4)[8]         &IV   &1&0&0&0&0&0&0\\
&           &  &2&1&0&0&1&0&0\\
&           &  &4&3&0&0\mbox{ or }1&3\mbox{ or }1&0&0\mbox{ or }1\\
&           &  &8&3&0&0\mbox{ or }1&3\mbox{ or }1&1&2 \mbox{ or } 3\\
\hline
8&  \P(1,1,1,1,2)[6]      &IV    &1&0&0&0&0&0&0\\
&          &   &2&1&0&0&1&0&0\\
&          &   &6&1&0&0&1&1&2\\
\hline
9&  \P(1,1,1,1,4,6)[2,12]       &IV    &1&0&0&0&0&0&0\\
&          &   &2&1&0&0&1&0&0\\
&          &   &3&2&0&0\mbox{ or }1&2\mbox{ or }0&0&0\mbox{ or }1\\
&          &   &4&3&0&0\mbox{ or }1&3\mbox{ or }1&0&0\mbox{ or }1\\
&          &   &6&5&0&0,1,2&5,3,1&0&0,1,2\\
&          &   &12&7&0&0,1,2,3&7,5,3,1&1&2,3,4,5\\
\hline
10& \P(1,1,1,1,2,2)[4,4]        & II    &1&0&0&0&0&0&0\\
&        &     &2&1&0&0&1&0&0\\
&        &     &4&1&0&0&1&1&1\\
&        &     &8&5&1&0&3&2&2\\
\hline
11& \P(1,1,1,2,2,3)[4,6]        &IV    &1&0&0&0&0&0&0\\
&         &   &2&1&0&0&1&0&0\\
&         &    &12&7&0&0,1,2,3&7,5,3,1&1&2,3,4,5\\
\hline
12&   \P(1,1,1,1,1,2)[3,4]       &IV   &1&0&0&0&0&0&0\\
&          &   &2&1&0&0&1&0&0\\
&          &   &3&2&0&0\mbox{ or }1&2\mbox{ or }0&0&0\mbox{ or }1\\
&          &   &12&7&0&0,1,2,3&7,5,3,1&1&2,3,4,5\\
\hline
13&    \P(1,1,2,2,3,3)[6,6]     & II      &1&0&0&0&0&0&0\\
&       &   &2&1&0&0&1&0&0\\
&       &      &3&2&0&0\mbox{ or }1&2\mbox{ or }0&0&0\mbox{ or }1\\
&       &      &6&3&0&0\mbox{ or }1&3\mbox{ or }1&1&1\mbox{ or }2\\
\hline
14&  \P(1,1,1,1,1,3)[2,6]        & I   &1&0&0&0&0&0&0\\                                  
&          &     &3&2&0&0\mbox{ or }1&2\mbox{ or }0&0&0\mbox{ or }1\\
&          &   &6&2&0&0\mbox{ or }1&2\mbox{ or }0&1&2\mbox{ or }3\\
\hline
\end{array}
\]
\end{small}

\section{Rohde's example}

In \cite{rohde,bert} one finds examples of one-dimensional families $f: X \to S$ of certain Calabi-Yau $3$-folds. 
Their construction is induced by a Borcea-Voisin method, i.e., is obtained from a product of a fixed elliptic curve $E$ and a K3 surface $S_\lambda$
by application of certain automorphisms. To describe the underlying VHS, in 
section 2 of \cite{bert} a family of genus two Picard curves $C_\lambda$ is constructed, given by a triple covering $C_\lambda \to \P^1$, and thus coming
with an automorphism $\xi$ of order three. 
The cohomology $H^1(C_\lambda,\Q)$ has an eigenspace decomposition according to the eigenvalues $\xi$ and $\bar{\xi}=\xi^2$ and 
it is strongly related to the cohomology of the fibers of $f$. Namely, one has 
$$
H^{3,0}(X_\lambda,\Q)= H^{1,0}(C_\lambda)_{\bar{\xi}}, \; H^{2,1}(X_\lambda,\Q)= H^{0,1}(C_\lambda)_{\bar{\xi}}, 
$$
$$
H^{1,2}(X_\lambda,\Q)= H^{1,0}(C_\lambda)_{\xi}, \; H^{0,3}(X_\lambda,\Q)= H^{0,1}(C_\lambda)_{\xi}.
$$
Furthermore, the family $C_\lambda$ is induced from a Shimura family, see \cite{bert}. As a consequence, the Higgs map $\vartheta$ induces non-zero maps
$$
\vartheta: E^{3,0} \longrightarrow E^{2,1} \otimes \Omega^1_{\bar{S}}(\log D), \;
\vartheta: E^{1,2} \longrightarrow E^{0,3} \otimes \Omega^1_{\bar{S}}(\log D)
$$
induced by the corresponding Higgs fields for the family $C_\lambda$, and the zero morphism 
$$
\vartheta: E^{2,1} {\buildrel 0 \over \longrightarrow}  E^{1,2} \otimes \Omega^1_{\bar{S}}(\log D), 
$$
by noting that Higgs fields respect eigenspace decompositions. 

In this case one knows a little bit more about $a$ and $b$: 
One has $\bar{S}=\P^1$ and $\sharp D=3$ singular points, one of them of type $IV$. Hence $|IV|=1$ and $|II|=2$ in our case.
This follows from \cite[Sec. 2]{bert} from the fact that the resulting Picard-Fuchs equation is a classical hypergeometric equation with singularities at $0,1,\infty$.
Let $F=F^{1,0}\oplus F^{0,1}$ be the Higgs bundle associated to the variation of the genus two curves $C_\lambda$. 
Then $F$ decomposes according to eigenspaces, i.e., $F=F_\xi \oplus F_{\bar \xi}$. 
Due to the existence of non-unipotent points, $F^{1,0}_\Box$ and $F^{0,1}_\Box$ for $\Box \in \{\xi,\bar{\xi}\}$ are not dual to each other. One has: 

\begin{lemma} In Rohde's example, each rank two Higgs bundle $F_\Box$ has a maximal Higgs field, i.e., 
$$
\vartheta: F^{1,0}_\Box {\buildrel \cong \over \longrightarrow} F^{0,1}_\Box \otimes \Omega^1_{\P^1}(\log D).
$$
is an isomorphism. Furthermore, $\deg F^{1,0}_\Box =0$ and $\deg F^{0,1}_\Box =-1$. 
\end{lemma}

\begin{proof} Thm. \ref{jzthm}, i.e., the Arakelov inequality of Jost and Zuo \cite[Thm. 1]{jz}, implies that $\deg F_\Box^{1,0} \le \frac{1}{2}$, hence $\deg F_\Box^{1,0} \le 0$. 
On the other hand, one has $\deg F_\Box^{1,0} \ge 0$: Consider the local system $\W_\Box$ corresponding to $F_\Box$. It satisfies $h^2(\P^1,j_*\W_\Box)=h^0(\P^1,j_*\W_\Box)=0$ 
by the argument of \cite[Prop 3.6.]{hanoi}. The Higgs complex for $F_\Box$ is given by
$$
\begin{xy}
\xymatrix{
F_\Box^{1,0}(-I)  \ar[rd]^{\vartheta}_{\neq 0} &   F_\Box^{0,1}   \\
F_\Box^{1,0}(II)\otimes  \Omega^1_{\P^1} &  F_\Box^{0,1}(II) \otimes  \Omega^1_{\P^1}  } 
\end{xy}
$$
as in the proof of Theorem \ref{m1}. 
Therefore, $H^1(\P^1,F_\Box^{1,0}(II)\otimes\Omega^1_{\P^1})$ is the direct summand of Hodge type $(2,1)$ inside $H^2(\P^1,j_*\W_\Box)=0$.
Since $|II|=1$, we obtain $0=h^1(\P^1,F_\Box^{1,0}(1)\otimes\Omega^1_{\P^1})=h^0(\P^1,(F_\Box^{1,0})^{-1}(-1))$. Therefore $F_\Box^{1,0}={\mathcal O}_{\P^1}$.
In a similar way, $H^0(F_\Box^{0,1})$ contributes to $H^0(\P^1,j_*\W_\Box)=0$, therefore $\deg F_\Box^{0,1}<0$. 
Since $\vartheta$ is a non-zero map, and $\deg(\Omega^1_{\P^1}(\log D))=1$, we get that $\deg F_\Box^{0,1}=-1$ and $F_\Box^{0,1}={\mathcal O}_{\P^1}(-1)$.
\end{proof}

\begin{cor} It follows that for the Higgs bundle $E$ one has $a=0$ and $b=-1$.
Furthermore, the Higgs maps 
$\vartheta: E^{3,0} \longrightarrow E^{2,1} \otimes \Omega^1_{\P^1}(\log D)$, and 
$\vartheta: E^{1,2} \longrightarrow E^{0,3} \otimes \Omega^1_{\P^1}(\log D)$ are both isomorphisms.
\end{cor}

Note that the identities for $a'=a+|IV|$, $b'=b+|IV|$ in Remark \ref{degreeremark} still hold. 

The properties of $E$ we have shown are summarized in the following definition. 

\begin{definition}
A logarithmic Higgs bundle $E=E^{3,0} \oplus E^{2,1} \oplus E^{1,2} \oplus E^{0,3}$ of weight $m=3$ and rank $4$ on $\bar{S}$ is called
decomposed, if $\vartheta: E^{3,0} \to  E^{2,1} \otimes \Omega^1_{\bar{S}}(\log D)$ and 
$\vartheta: E^{1,2} \to E^{0,3} \otimes \Omega^1_{\bar{S}}(\log D)$ are isomorphisms, and 
$\vartheta: E^{2,1} \to  E^{1,2} \otimes \Omega^1_{\bar{S}}(\log D)$ is the zero map. 
\end{definition}

\begin{thm} \label{m3bis}
The $L^2$-Higgs cohomology of a decomposed Higgs bundle $E=E^{3,0} \oplus E^{2,1} \oplus E^{1,2} \oplus E^{0,3}$ 
of weight $m=3$ and rank $4$ with $a +|IV| > 0$ is described as follows: 
$$
h^1_{L^2}(S,\V)^{(4,0)}=h^0(\bar{S},E^{3,0}(IV) \otimes \Omega^1_{\bar{S}})=g-1+a + |IV|, \; h^1_{L^2}(S,\V)^{(3,1)}= h^1_{L^2}(S,\V)^{(1,3)}=0, 
$$
$$
h^1_{L^2}(S,\V)^{(2,2)}=h^0(\bar{S},E^{1,2}(IV) \otimes \Omega^1_{\bar{S}}) \oplus h^1(\bar{S},E^{2,1}(-I-III))=2h^0(\bar{S},E^{1,2}(IV) \otimes \Omega^1_{\bar{S}}). 
$$
The assumptions imply that $|I|=|III|=0$ and $a=b+2g-2+\sharp D$. 
\end{thm}

\begin{proof}  We use the same notations for the $L^2$-Higgs complex $\Omega^0_{(2)}(E) {\buildrel \vartheta \over \to} \Omega^1_{(2)}(E)$ as above. 
The symmetry of decomposed Higgs bundles implies that $|I|=|III|=0$, since such degenerations cannot occur. 
As $E$ is decomposed, also the arrow $\Omega^0_{(2)}(E)^{2,1} \to \Omega^1_{(2)}(E)^{1,2}$ is still zero. Also the two non-zero arrows 
in the following diagram remain isomorphisms (which implies again that $|I|=0$):
$$
\begin{xy}
\xymatrix{
E^{3,0}(-II-III)  \ar[rd]_\cong^{\vartheta} &   E^{2,1}(-I-III)  \ar[rd]_{0}^{\vartheta}  & E^{1,2}(-II) \ar[rd]_\cong^{\vartheta}  &  E^{0,3} \\
E^{3,0}(IV)\otimes  \Omega^1_{\bar{S}} &  E^{2,1}(IV) \otimes  \Omega^1_{\bar{S}} & E^{1,2}(IV)\otimes  \Omega^1_{\bar{S}} &  E^{0,3}(III+IV) \otimes  \Omega^1_{\bar{S}} } 
\end{xy}
$$
It follows that $a=b+2g-2+|II|+|III|+|IV|$, and by Riemann-Roch, using the assumption $a +|IV| > 0$,
$$
h^1_{L^2}(S,\V)^{(4,0)}=h^1_{L^2}(S,\V)^{(0,4)}=h^0(\bar{S},E^{3,0}(IV) \otimes \Omega^1_{\bar{S}})=g-1+a +|IV|, 
$$
$$
h^1_{L^2}(S,\V)^{(3,1)}= h^1_{L^2}(S,\V)^{(1,3)}=0, 
$$
$$
h^1_{L^2}(S,\V)^{(2,2)}=h^0(\bar{S},E^{1,2}(IV) \otimes \Omega^1_{\bar{S}}) \oplus h^1(\bar{S},E^{2,1}(-I-III)). 
$$
In the last line, the two summands are dual to each other, which implies again $|I|=|III|=0$, and 
$h^1_{L^2}(S,\V)^{(2,2)}=2 h^0(\bar{S},E^{1,2}(IV) \otimes \Omega^1_{\bar{S}})$.  
\end{proof}

Theorem \ref{m3bis} implies: 

\begin{cor} In Rohde's example one has $h^1_{L^2}(S,\V)=0$, consequently all Hodge numbers vanish: 
$$
h^1_{L^2}(S,\V)^{(4,0)}=h^1_{L^2}(S,\V)^{(0,4)}=h^1_{L^2}(S,\V)^{(3,1)}= h^1_{L^2}(S,\V)^{(1,3)}=h^1_{L^2}(S,\V)^{(2,2)}=0.
$$
\end{cor}

\medskip
In particular, since $|I|=|III|=0$ and $|II|=2$, $|IV|=1$, the check sum 
$$
h^1(j_*\V)=h^{4,0}+h^{3,1}+h^{2,2}+h^{1,3}+h^{0,4}=8g-8 + |I| + 2|II| + 3|III|+ 4|IV|=0
$$
is correct. Base change maps $e: \P^1 \to \P^1$ with prescribed ramification lead to more families where the theorem can be applied. 
Details can be found in the forthcoming thesis of Henning Hollborn \cite{hollborn}. \\

{\bf Acknowledgement:} This work builds up on \cite{hanoi} in an essential way, and we would like to thank P. L. del Angel, 
D. van Straten, and K. Zuo for the previous collaboration. We thank Xuanming Ye for several additional discussions, and the referee
for some valuable improvements. This work was supported by DFG Sonderforschungsbereich/Transregio 45.


\begin{thebibliography}{AMSZ}
\bibitem{AESZ} G. Almkvist, C. van Enckevort, D. van Straten, W. Zudilin: Tables of Calabi-Yau equations, unpublished, 
{\tt arXiv:math/0507430}.
\bibitem{hanoi} P. L. del Angel, S. M\"uller-Stach, D. van Straten, K. Zuo: 
Hodge classes associated to $1$-parameter families of Calabi-Yau $3$-folds, 
Acta Mathematica Vietnamica, Vol. 35, 1-16 (2010). 
\bibitem{coxzucker} D. Cox, S. Zucker: Intersection numbers of sections of elliptic surfaces, 
Inventiones Math. Vol. 53, 1-44 (1979). 
\bibitem{bert} A. Garbagnati, B. van Geemen: The Picard-Fuchs equation of a family of Calabi-Yau threefolds without maximal unipotent monodromy,
International Math. Res. Notices, 3134-3143 (2010). 
\bibitem{hollborn} H. Hollborn: Thesis, Univ. of Mainz, in preparation.
\bibitem{jz} J. Jost, K. Zuo: Arakelov type inequalities for Hodge bundles over algebraic varieties,
Journal of Algebraic Geometry, Vol. 11, 535-546 (2002).
\bibitem{jyz} J. Jost, Y.-H. Yang, K. Zuo: Cohomologies of unipotent harmonic bundles over noncompact curves, 
Crelle's Journal, Vol. 609, 137-159 (2007).
\bibitem{morrison} D. Morrison, J. Walcher: D-branes and normal functions, Adv. Theor. Math. Phys., Vol. 13, 553-598 (2009). 
\bibitem{rohde} J.-C. Rohde: Maximal automorphisms of Calabi-Yau manifolds versus maximally unipotent monodromy, Manuscripta Mathematica, 
Vol. 131, 459-474 (2010), 
\bibitem{s} C. Simpson: Constructing variations of Hodge structure using Yang-Mills theory and applications to uniformization,
Journal of the American Mathematical Society, Vol. 1, 867-918 (1988).
\bibitem{stiller} P. Stiller: Automorphic forms and the Picard number of an elliptic surface, Aspects of Mathematics, Vol. E5, Vieweg Verlag (1984).
\bibitem{vz} E. Viehweg, K. Zuo: On the isotriviality of families of projective manifolds over curves, Journal of Algebraic Geometry, Vol. 10, 781-799 (2001).
\bibitem{zucker} S. Zucker: Hodge theory with degenerating coefficients,
Annals of Mathematics, Vol. 109, 415-476 (1979). 
\end{thebibliography}
\end{document}